\documentclass{article}

\usepackage{amsfonts}
\usepackage{amsmath}
\usepackage{amsthm}
\usepackage{color}

\newtheorem{theorem}{Theorem}
\newtheorem{conjecture}[theorem]{Conjecture}
\newtheorem{question}[theorem]{Question}

\title{The $Q_2$-free process in the hypercube} 
\author{J.~Robert Johnson\footnote{School of Mathematical Sciences, Queen Mary
University of London, London E1 4NS, England} \footnote{\tt r.johnson@qmul.ac.uk}\and Trevor Pinto$^*$ \footnote{\tt t.pinto@hotmail.co.uk}}

\begin{document}

\maketitle
\begin{abstract}
The generation of a random triangle-saturated graph via the triangle-free process has been studied extensively. In this short note our aim is to introduce an analogous process in the hypercube. Specifically, we consider the $Q_2$-free process in $Q_d$ and the random subgraph of $Q_d$ it generates. Our main result is that with high probability the graph resulting from this process has at least $cd^{2/3} 2^d$ edges. We also discuss a heuristic argument based on the differential equations method which suggests a stronger conjecture, and discuss the issues with making this rigorous. We conclude with some open questions related to this process.
\end{abstract}

\section{Introduction}

Let $F$ be a (typically small) graph. A graph $G$ on vertex set $V$ is \emph{$F$-saturated} if it contains no copy of $F$ as a subgraph but the addition of any new edge in $V^{(2)}\setminus E(G)$ creates a copy of $F$. The $F$-free process is a well-known way of generating a random $F$-saturated graph. We fix a finite set $V$ and form a nested sequence $G_0,G_1,\dots,G_M$ of $F$-free graphs with common vertex set $V$. For each $i$ we have $|E(G_i)|=i$ and $G_{i+1}$ is obtained from $G_i$ by randomly adding a new edge chosen uniformly at random from all of the possible edges which do not create a copy of $F$. The process stops when no new edge can be added; in other words $G_M$ is an $F$-saturated graph with $M$ edges. We can now ask: what can be said about the properties of the graph $G_M$, and in particular the random variable $M$?

Work in this direction was initiated in 1992 by Ruci\'nski and Wormald \cite{rucinskiwormald}, who studied the case $F=K_{1,3}$, the star with three leaves, investigating the structure of $G_M$.

A major breakthrough in the area was the 2009 paper of Bohman \cite{bohman} on the case $F=K_3$, the so-called \emph{triangle-free process}. Using the differential equations method for random graph processes introduced by Rucinski and Wormald in \cite{rucinskiwormald} (see for instance \cite{wormald} for a survey of the subject), Bohman determined the order of $M$ with high probability.

\begin{theorem}[Bohman \cite{bohman}]
Let $G_M$ be the graph generated by the triangle-free process with $|V|=n$. Then with high probability,
\[c_1(\log n)^{\frac{1}{2}} n^{\frac{3}{2}} \leq M \leq c_2(\log n)^{\frac{1}{2}} n^{\frac{3}{2}},\]
for some constants $c_1$ and $c_2$.
\end{theorem}  

This result was later refined by Pontiveros, Griffiths and Morris \cite{pontiverosetal} and independently by Bohman and Keevash \cite{bohmankeevash}. Both sets of authors used a substantial extension of the differential equations method to determine $M$ asymptotically, with high probability. They also used their analysis of the triangle-free process to improve the known bounds on the Ramsey number $R(3,t)$.

There is now a large body of work on the $F$-free process for other graphs. See for instance \cite{bohmankeevash2, picollelli, warnke} and the references therein.

Now let $H$ be a (typically large) \emph{host graph}. A graph $G$ is \emph{$(H,F)$-saturated} if it is an $F$-free subgraph of $H$ but the addition of any new edge in $E(H)\setminus E(G)$ creates a copy of $F$. With this formulation the usual notion of $F$-saturation corresponds to $(K_n,F)$-saturation. Our particular interest is the case where the host graph $H$ is the hypercube $Q_d$. Where $Q_d$ is defined to have vertex set $\{0,1\}^d$, with two vertices joined by an edge if they differ in exactly one coordinate. 

Analogously to the usual triangle-free process, we define the \emph{$(Q_d,Q_2)$-free process} as the random nested sequence of subgraphs of $Q_d$ generated by repeatedly adding a new edge chosen uniformly at random from all those edges of $Q_d$ which do not create a copy of $Q_2$. We consider $Q_2$ as our forbidden subgraph since this is the most natural analogue of the triangle-free process (forbidding $K_3$ in $K_n$). However, the definition naturally extends to other forbidden subgraphs of $Q_d$. For comparison, the unconstrained random process with the hypercube as host graph is considered in \cite{AKS,BKL,erdosspencer}.

We will describe this process more formally using an equivalent definition based on random permutations which turns out to be easier to work with. The approach was first used (in the $K_n$ host graph setting) by Erd\H os, Suen and Winkler \cite{esw} and has also appeared in a more refined form, for instance in \cite{makai,warnke2}. We first choose a uniformly random permutation of $E(Q_d)$ giving a labelling of these edges as $e_1,e_2,\dots,e_{|E(Q_d)|}$. From this, form a nested sequence of subgraphs of $Q_d$ by looking at the edges in turn, and adding the next edge which does not create a copy of $Q_2$. More precisely, suppose that we have constructed graphs $G_0,G_1,\dots, G_i$ and have looked at edges $e_1,\dots,e_{t(i)}$. We look at the edges $e_{t(i)+1},e_{t(i)+2},\dots$ in turn, stopping when we get to some $e_j$ which can be added to $G_i$ without creating a copy of $Q_2$. We add $e_j$ to $G_i$ to form $G_{i+1}$ and let $t(i+1)=j$. The result is a nested sequence of subgraphs $G_0,G_1,\dots,G_M$ of $Q_d$ with $E(G_i)=i$ and $G_M$ being $(Q_d,Q_2)$-saturated. 

As in the graph case, our main question is what can be said about the random variable $M$? Our main result, proved in Section 2, is that with high probability, the subgraph of $Q_d$ generated by the $Q_2$-free process in $Q_d$ has at least $cd^{2/3} 2^d$ edges, for some constant $c$. We also establish a local version of this result: with high probability, almost all vertices have degree at least $cd^{2/3}$. 

In Section 3 we consider the differential equations heuristic in the hypercube context. This leads to the conjecture that the order of magnitude of the number of edges in the graph generated is $(\log d)^{1/3} d^{2/3} 2^d$. This conjecture is a consequence of the guiding heuristic for other constrained processes that the process should resemble an unconstrained graph process except in statistics involving the constraint. Indeed, let $O_i$ be the number of `open pairs' in $G_i$ (ie edges of $Q_d$ that can be added that can be added to $G_i$ without producing a copy of $Q_2$). Clearly, for all $i$, we have that $i+O_i$ is an upper bound for the size of the final graph. If the heuristic is correct for our process then $O_i$ should be close to $(1-p^3)^{d-1}d2^{d-1}$. Under this assumption, when $i=O\left((\log d)^{1/3} d^{2/3} 2^d\right)$ the upper bound $i+O_i$ is also $O\left((\log d)^{1/3} d^{2/3} 2^d\right)$ and so this will be the size of the final graph. We also see that the proof strategy used to make the differential equations method rigorous followed in the ordinary graph triangle-free process is unlikely to work in the hypercube context without some significant new ideas. Thus the problem of analysing this process may be a natural testing ground for extending the differential equations method further or developing new techniques. 

Related to this heuristic, we now describing how our process relates to an important general framework for constrained graph process introduced by Bennett and Bohman \cite{benboh}. This is based on independent sets in an auxiliary hypergraph. For the usual $F$-free process this auxiliary hypergraph $H$ has \emph{vertex set} $V(H)=E(K_n)$, with a set $\{e_1,\dots, e_r\}\subset V(H)$ forming an edge of $H$ precisely when $e_1,\dots, e_r$ form a copy of $F$ when viewed as edges of $K_n$. Now, maximal independent sets in $H$ correspond to $F$-saturated graphs. The $F$-free process corresponds to the natural random greedy algorithm for finding a maximal independent set in $H$. Namely, we grow our independent set by adding a new random vertex of $H$ at each time step subject to the condition that the set remains independent (ie contains no edge of $H$). Bennett and Bohman's result gives a lower bound of $\Omega\left( N\left(\frac{\log N}{D}\right)^{\frac{1}{r-1}}\right)$ on the independent set generated by this random greedy algorithm in an $r$-uniform, $D$-regular hypergraph with $N$ vertices, subject to certain degree and co-degree conditions. One of these conditions is that $H$ is reasonably dense in the sense that $D>N^{\epsilon}$. This bound is consistent with the heuristic mentioned above that the independent set generated after $i$ steps should resemble a subset generated by independently picking each element of $V$ with probability $p=i/N$ except in that it contains no edges of $H$. 

Our $(Q_d,Q_2)$-free process can also be expressed in this framework. The auxiliary hypergraph has $N=d2^{d-1}$, $r=4$ (each $Q_2$ consists of $4$ edges) and $D=d-1$ (each edge of $Q_d$ lies in $d-1$ copies of $Q_2$). Notice that $D$ is only $O(\log N)$ and so this hypergraph is too sparse to apply the Bennett and Bohman result. The $(Q_d,Q_2)$-free process then provides a natural example of a constrained graph process to which the Bennett and Bohman result does not apply. 

We conclude, in Section 4, by raising some related open problems.

Finally, we note that, complementing these probabilistic questions, saturated graphs have been studied from an extremal perspective. Indeed, the well-studied \emph{Tur\'an number} of $F$, denoted by $ex(n,F)$, can be defined as the maximum number of edges in an $F$-saturated graph on $n$ vertices. As a counterpart to this, the \emph{saturation number} of $F$, denoted by sat$(n,F)$ is the minimum number of edges in an $F$-saturated graph. See the surveys \cite{furedisimonovits} and \cite{faudrees,pikhurko} and many references therein for more on Tur\'an and saturation numbers. Both Tur\'an and saturation numbers have been studied for the host graph $Q_d$ (see for instance \cite{AlonKS,BHLL,erdos,johnsontalbot} for the former and \cite{GK,johnsonpinto,MNS} for the later). However, to our knowledge, this associated random process has not.

\section{Main Result}

\begin{theorem}\label{process}
Let $M$ be the number of edges in the subgraph $G_M$ of $Q_d$ generated by the $(Q_d,Q_2)$-free process. With high probability, $M>cd^{2/3} 2^d$, for some constant, $c$.
\end{theorem}

As we shall see, the constant $c$ can be taken to be arbitrarily close to $1/e$.

The proof uses the random permutation formulation of the process. We identify a local condition on the permutation which guarantees that a particular edge appears in the final graph $G_M$. Calculating the probability that this condition is satisfied gives a lower bound on the expected number of edges. The fact that the condition is a local one means that dependence between edges is limited and the second moment method gives a lower bound on $M$ which holds with high probability. 

\begin{proof}[Proof of Theorem \ref{process}]
Generate a random permutation of the edges of $Q_d$ by assigning to each edge $e$, a random variable $T_e$, where $T_e$ is uniformly distributed in the interval $[0,1]$. We say that $e$ precedes $f$ in our order if $T_e<T_f$.

Let $G_M$ be the saturated graph yielded by following the $(Q_d,Q_2)$-free process on this permutation.

We say that an edge $e\in E(Q_d)$ is \emph{good} if, for every $Q_2$ containing $e$, the last of its four edges in our ordering is not $e$. It is easy to see that if $e$ is good, then $e$ is an edge of $G_M$. 

Let $A_e$ denote the indicator random variable taking the value 1 if $e$ is good and 0 otherwise, and let $A=\sum_{e\in E(Q_d)} A_e$ be the total number of good edges. Considering how the permutation is generated from the variables $T_e$, we obtain:
\begin{align*}
\mathbb{P}(A_e=1)&=\int_0^1 (1-x^3)^{d-1} dx\\
&\geq \int_0^{d^{-\frac{1}{3}}} (1-x^3)^{d-1} dx\\
&\geq d^{-\frac{1}{3}}\left(1-\frac{1}{d}\right)^{d-1} \quad\text{(as the integrand is decreasing in $x$)}\\
\mathbb{P}(A_e=1)&\geq \left(1/e+o(1)\right)d^{-\frac{1}{3}},
\end{align*}
for large enough $d$.
 
Since $Q_d$ has $d2^{d-1}$ edges, linearity of expectation gives that: 
\[
\mathbb{E}(A)\geq(1/e+o(1)) d^{2/3}2^{d-1}.
\]

The event $A_e$ depends only on the variables $T_f$ where $f$ is one of the $3(d-1)$ edges contained in a $Q_2$ through $e$. It follows that $A_e$ is independent of all but at most $9d^2$ other $A_f$.   

\begin{align*}
\text{Var}(A)&=d2^{d-1}\text{Var}(A_e)+\sum_{e\not=f} Cov(A_e, A_f)\\
&\leq \mathbb{E}(A)+ 9d^3 2^d\\
&=o(\mathbb{E}(A)^2).
\end{align*}
Thus by Chebychev's inequality, $A\geq c d^{2/3}2^d$ with high probability, for some $c$  (which can be taken to be arbitrarily close to $1/e$). This concludes the proof, since $M\geq A$.
\end{proof}

A slightly more careful calculation gives that $d^{2/3}2^{d-1}$ is the correct order of magnitude of $\mathbb{E}(A)$ so $A=\Theta(d^{2/3}2^{d-1})$ with high probability. However, because the property of being good is sufficent but not necessary for an edge to be in $G_M$, this observation gives no upper bound for $M$. 

Notice that the way in which we bounded the integral, means that only edges with $T_e\leq d^{-1/3}$ are considered. This means that with high probability, not only do we finish the process with at least $c d^{2/3}2^d$ edges, but at least this many edges must be added from among the first $d^{2/3}2^{d-1}$ edges considered. 

The same approach can be used to give some information on the degrees in $G_M$. The degree of $v$ in $G_M$ is bounded by the number of good edges among the edges of $Q_d$ incident to $v$. Unfortunately, there are no independent pairs among the events that each of these edges is good. However, the dependence is very limited and so a local analogue to Theorem \ref{process} can be established.  

\begin{theorem}\label{degree}
Let $v$ be a fixed vertex of $Q_d$ and let $G_M$ be the subgraph of $Q_d$ generated by the $(Q_d,Q_2)$-free process. With high probability, the degree of $v$ in $G_M$ is at least $cd^{2/3}$, for some constant, $c$.
\end{theorem}

\begin{proof}
Let $e_1,\dots,e_d$ be the edges of $Q_d$ incident to $v$. We define good edges as in the proof of Theorem \ref{process}. Let $A_i$ be the indicator variable of the event `$e_i$ is good', and $D_v=\sum_{i=1}^d A_i$. We have that
\[
\text{Var}(D_v)=d\text{Var}(A_e)+\sum_{i\not=j} \text{Cov}(A_i,A_j)=dp(1-p)+d(d-1)(r-p^2)
\]
where
\begin{align*}
p&=\mathbb{P}(A_1=1)=cd^{-1/3}\\
r&=\mathbb{P}(A_1=1,A_2=1)
\end{align*}

It will suffice to show that $r-p^2=o(d^{-2/3})$. From this we deduce that $\text{Var}(D_v)=o(d^{4/3})=o(\mathbb{E}(D_v)^2)$, and the result follows as in the proof of Theorem \ref{process}. 

Now, using the same method of generating a random permutation as in the proof of Theorem \ref{process}, we have:
\begin{align*}
p^2&=\int_0^1\int_0^1 f(x,y,d)\, dx\, dy\\
r&=\int_0^1\int_0^1 g(x,y,d)\, dx\, dy
\end{align*}
where
\begin{align*}
f(x,y,d)&=(1-x^3)^{d-1}(1-y^3)^{d-1}\\
g(x,y,d)&=(1-x^3-y^3+x^2y^2\min\{x,y\})^{d-2}(1-(\max\{x,y\})^2).
\end{align*}

It is easy to check that if one of $x,y$ is greater than $d^{-1/4}$ then both $f(x,y,d)$ and $g(x,y,d)$ are at most $\exp(d^{-1/4})$ and so the contribution to $r-p^2$ from this range of $x,y$ is certainly $o(d^{-2/3})$.

On the other hand, if $x<y<d^{-1/4}$ then, writing 
\[
g(x,y,d)-f(x,y,d)=(1-x^3)^{d-1}(1-y^3)^{d-1}\left(h(x,y)-1\right)
\]
where
\begin{align*}
h(x,y)&=\left(\frac{1-x^3-y^3+x^3y^2}{(1-x^3)(1-y^3)}\right)^{d-2}\left(\frac{1-y^2}{(1-x^3)(1-y^3)}\right)\\
&=\left(1+\frac{x^3y^2-x^3y^3}{(1-x^3)(1-y^3)}\right)^{d-2}\left(1+\frac{x^3+y^3-x^3y^3-y^2}{(1-x^3)(1-y^3)}\right)\\
&<\left(1+(d-2)2d^{-5/4}\right)\left(1+2d^{-3/4}\right)\\
&<\left(1+3d^{-1/4}\right),
\end{align*}

we conclude that
\[
\int_0^{d^{-1/4}}\int_0^{d^{-1/4}} g(x,y,d)-f(x,y,d)\, dx\, dy\leq 3(d^{-1/4})^3.
\]
That is, the contribution to $r-p^2$ from this range of $x,y$ is at most $3d^{-3/4}=o(d^{-2/3})$ as required.
\end{proof}

We do not know whether or not, with high probability \emph{all} vertices have degree at least $cd^{2/3}$.

\section{Heuristic}

For the triangle-free process, Bohman \cite{bohman} introduces a heuristic that assumes certain random variables follow some trajectories closely. Using this assumption he deduces the values of those trajectories in order to bound the number of edges in the resulting graph. This approach can be made rigorous using martingales. 

We use the analogous heuristic for the $(Q_d,Q_2)$-free process to suggest a possible order for $M$. However, we also point out some differences between the $(Q_d, Q_2)$-free process and the triangle-free process that cause difficulties in making this argument rigorous.

Let $G_0, \dots, G_M$ be the sequence of graphs generated by the $(Q_d, Q_2)$-free process. Let $u$ and $v$ be a pair of vertices that are adjacent in $Q_d$. We say that $uv$ is open in $G_i$ if there is no path of three $G_i$-edges that connect $u$ to $v$. In other words $uv$ is open if adding it to $G_i$ does not form a copy of $Q_2$. We write $O_i$ for the number of open pairs in $G_i$.  This definition of open pairs is analogous to a definition in \cite{bohman}.  

We also define, for each $Q_d$-adjacent pair of vertices $u$ and $v$, three other random variables. Let $W_i(uv)$ denote the number of paths of length 3 from $u$ to $v$ consisting of three open pairs in $G_i$, let $X_i(uv)$ be the number of paths of length 3 from $u$ to $v$ consisting of two open pairs and one $G_i$-edge and let $Y_i(uv)$ count the paths of length 3 from $u$ to $v$ consisting of one open pair and two $G_i$-edges.

For convenience, we also introduce a scaling $t=\frac{i}{d^{2/3}2^d}$. We assume there are continuous functions $q, w, x$ and $y$ such that for all $i$ and all $Q_d$-adjacent $u$ and $v$:
\[
O_i\approx q(t)d2^d, \quad W_i(uv)\approx w(t)d, \quad X_i(uv)\approx x(t)d^{2/3}, \quad Y_i(uv)\approx y(t)d^{1/3}.
\]

Note that adding a single edge $uv$ to $G_i$ to form $G_{i+1}$ removes $Y_i(uv)$ open edges. Thus for small $\epsilon$, we expect
\[
q(t+\epsilon)d2^d  \approx O_{i+\epsilon d^{2/3}2^d}\approx O_i- \epsilon d^{2/3}2^d\cdot y(t)d^{1/3}\approx (q(t)-\epsilon y(t))d2^d .
\]

This suggests that $\frac{dq}{dt}=-y$. Similar arguments give:

\[\frac{dx}{dt}=\frac{3w}{q}-\frac{2xy}{q},\qquad \frac{dy}{dt}=\frac{2x}{q}-\frac{y^2}{q},\qquad\frac{dw}{dt}=\frac{-3yw}{q}.\]

Solving these equations with initial conditions $q(0)=1/2,w(0)=1,x(0)=y(0)=0$ gives 
\[
q(t)=\frac{1}{2}e^{-8t^3}, \quad w(t)=e^{-24t^3}, \quad x(t)=6te^{-16t^3}, \quad y(t)=12t^2e^{-8t^3}.
\]

For any $i$, the final number of edges in the process is bounded from above by $i+O(i)$ and from below by $i$. If indeed $O(i)\approx \frac{1}{2}e^{-8t^3} d2^d$, then when $t=\Theta(\log^{1/3} d)$, these bounds are both $\Theta\left( (\log d)^{\frac{1}{3}} d^{\frac{2}{3}} 2^d\right)$.  

As we noted in the introduction, this solution fits with the heuristic that a constrained process should resemble an unconstrained process except in statistics involving the constraint. 

Motivated by this heuristic, we propose the following,

\begin{conjecture}
Let $G_M$ be the graph generated by the $(Q_d,Q_2)$-free process. With high probability,
\[c_1(\log d)^{1/3} d^{2/3} 2^d \leq M \leq c_2(\log d)^{1/3} d^{2/3} 2^d,\]
for some constants $c_1$ and $c_2$.
\end{conjecture}

For the triangle-free process, Bohman uses martingales to show that with high probability all the relevant random variables do indeed follow their trajectories closely. By contrast, in our process the situation is more complicated; the random variables we use to track the evolution of the graph do not all follow the trajectory indicated by the differential equations heuristic.

Associated with the $(Q_d, Q_2)$-free process, we have a sequence of graphs, $H(j)$, for $j=0, \dots, n2^{n-1}$, where $H(j)$ is the graph formed by the first $j$ edges in the randomly chosen permutation.  This nested sequence of graphs is a natural analogue of the unconstrained Erd\H{o}s-Renyi random graph process 

We again let $G_i$ denote the graphs of the $(Q_d, Q_2)$-free process for $i=0,\dots, M$, but consider $i$ as a function of $j$. That is, we write $i(j)$ for the number of edges added from among the first $j$ edges looked at. For $Q_d$-adjacent vertices $u$ and $v$, note that $Y_{i(j)}(uv)=0$  whenever $u$ and $v$ are isolated in $H(j)$. Thus,
\begin{align*}
\mathbb{P}(Y_{i(j)}(uv)=0)&\geq \frac{\binom{d2^{d-1}-2d}{j}}{\binom{d2^{d-1}}{j}}\\
&= \frac{(d2^{d-1}-j) \dotsm (d2^{d-1}-2d+1-j) }{(d2^{d-1})\dotsm(d2^{d-1}-2d+1)}\\
&\geq \left(1-\frac{j+d}{d2^{d-1}}\right)^{2d-1}\\
&\geq \exp\left(-\frac{j+d}{2^{d-2}}\right).
\end{align*}
 
It follows that there is some constant $c$ such that while $j\leq cd2^{d-1}$, we have, in expectation, a large number of pairs $uv$ with $Y_{i(j)}(uv)=0$. It seems likely that $i$ is approximately concave as a function of $j$ (the number of edges added should grow faster early on in the process when fewer edges have been looked at). If true this would imply that for some $uv$, the random variable $Y_{i(j)}(uv)$ equals zero for a constant proportion of the process. Thus, unlike in the triangle-free process, we will not typically have every variable following its expected trajectory closely. It is still possible that this approach can be salvaged, for instance by showing that the collection of variables are approximated by their trajectories in some weaker sense, but this does not appear to be straightforward.
 
\section{Further Questions}

Given the apparent obstacles to adapting the techniques from the triangle-free process to the hypercube, the main open problem is to develop tools to understand the $Q_2$-free process in $Q_d$. This could involve either refining the differential equations method or introducing a completely new approach.

The most immediate open problem is to give a good upper bound for $M$ and in particular to answer the following question:

\begin{question}
Is the true order of magnitude of $M$ given by $(\log d)^{1/3} d^{2/3} 2^d$ as predicted by the differential equations heuristic?
\end{question}

The number of edges is just one of the properties of $G_M$ that could be considered. It would be interesting to study other properties of this graph, for example the minimum and maximum degree of $G_M$. Note that Theorem \ref{degree} does not give any information about the minimum degree since it bounds almost all degrees rather than every degree.  

Bohman proves that with high probability the triangle-free process produces a graph with no large independent set. This was used to give improved lower bounds on the Ramsey number $R(3,k)$. In the hypercube there is no analogous Ramsey result; indeed for any $d$ there is a 2-colouring of $e(Q_d)$ with no monochromatic $Q_2$. Nevertheless, one could ask about the existence of empty subcubes.

\begin{question}
What can be said about the number of copies of $Q_k$ in $Q_n$ which contain no edges of $G_M$? For which $k$ is the expected number of empty $Q_k$ bounded away from 0?
\end{question}

More generally, what can said about the appearance of fixed subgraphs in $G_M$? This question has been addressed for the triangle-free process by Wolfovitz \cite{wolfovitz} (sparse subgraphs) and Gerke and Makai \cite{gerkemakai} (dense subgraphs). Extensions of their results to the $H$-free process were proved by Warnke \cite{warnke3}.

Finally, we studied the $Q_2$-free process as a natural special case of the $F$-free process. What can be said about the $F$-free process in the hypercube for other fixed graphs $F$? Two particularly appealing instances for $F$ are fixed dimension subcubes $Q_k$ and the star $K_{1,t}$ (the bounded degree process). 

\begin{question}
what can be said about the graph $G_M$ generated by the $F$-free process in $Q_d$? In particular when $F=Q_k$ or $F=K_{1,t}$.
\end{question}

\section{Acknowledgements}

We thank Lutz Warnke and an anonymous referee for helpful comments and suggestions.

\end{document}